\documentclass{amsart}
\usepackage{amssymb}
\usepackage{amsfonts}
\usepackage{amssymb}
\usepackage{amsmath}
\usepackage{amsthm}
\usepackage{enumerate}
\usepackage{pdfpages}
\usepackage{tikz}
\usepackage{amsmath}
\usepackage{amssymb}
\usetikzlibrary{arrows.meta}
\usepackage{tabularx}
\usepackage{centernot}
\usepackage{dutchcal}
\usepackage{mathtools}
\usepackage{stmaryrd}
\usepackage{amsthm,amssymb}
\usepackage{etoolbox,color}
\usepackage{tikz}
\usepackage{amssymb}
\usetikzlibrary{matrix}
\usepackage{tikz-cd}
\usepackage{tikz}
\usepackage{marginnote}
\definecolor{mygray}{gray}{0.85}
\usepackage[linecolor=black, backgroundcolor=mygray,colorinlistoftodos,prependcaption,textsize=small]{todonotes}
\usepackage{xargs}  
\usepackage[colorlinks,citecolor=blue,urlcolor=blue, linkcolor=blue]{hyperref}

\renewcommand{\leq}{\leqslant}
\renewcommand{\geq}{\geqslant}

\newcommand{\mrm}[1]{\mathrm{#1}}

\usepackage[backgroundcolor=mygray,colorinlistoftodos,prependcaption,textsize=small]{todonotes}
\usepackage{xcolor}

\makeatletter
\def\subsection{\@startsection{subsection}{3}%
  \z@{.5\linespacing\@plus.7\linespacing}{.3\linespacing}%
  {\bfseries\centering}}
\makeatother

\makeatletter
\def\subsubsection{\@startsection{subsubsection}{3}%
  \z@{.5\linespacing\@plus.7\linespacing}{.3\linespacing}%
  {\centering}}
\makeatother

\makeatletter
\def\myfnt{\ifx\protect\@typeset@protect\expandafter\footnote\else\expandafter\@gobble\fi}
\makeatother

\newtheorem{theorem}{Theorem}[section]

\newtheorem{corollary}[theorem]{Corollary}

\newtheorem{hypothesis}[theorem]{Hypothesis}
\newtheorem{question}[theorem]{Question}

\newtheorem{mclaim}[theorem]{Major Claim}
\newtheorem{claim}[theorem]{Claim}

\theoremstyle{plain}

\newcommand{\gtp}{\mathrm{ortp}}
\newcommand{\gS}{\mathcal{S}}

\theoremstyle{definition}

\newtheorem{fact}[theorem]{Fact}
\newtheorem{definition}[theorem]{Definition}
\newtheorem{remark}[theorem]{Remark}

\newtheorem{observation}[theorem]{Observation}
\newtheorem{notation}[theorem]{Notation}

\newcommand{\sprec}{\preccurlyeq_{\mathfrak{K}}}

\begin{document}

\begin{abstract}
It is an open problem of Mazari-Armida whether every abstract elementary class of $R$-modules $(\mathbf{K}, \leq_{\mathrm{pure}})$, with $\leq_{\mathrm{pure}}$ the pure submodule relation, is stable. We answer this question in the negative by constructing unstable abstract elementary classes $(\mathbf{K}, \leq_{\mathrm{pure}})$ of torsion-free abelian groups. On the other hand, we prove (in $\mathrm{ZFC}$) that if $R$ is any ring and $(\mathbf{K}, \preccurlyeq)$ is an abstract elementary class of $R$-modules which is $\kappa$-local (also called $\kappa$-tame) for some $\kappa \geq \mathrm{LS}(\mathbf{K}, \preccurlyeq)$, then $(\mathbf{K}, \preccurlyeq)$ is {\em almost stable}, where almost stability is a new notion of independent interest that we introduce in this paper, and which is equivalent to the usual notion of stability under the assumption of amalgamation. As a consequence, assuming the existence of a strongly compact cardinal $\kappa$, we have that every abstract elementary class $(\mathbf{K}, \preccurlyeq)$ of $R$-modules with amalgamation satisfying $\kappa > \mathrm{LS}(\mathbf{K}, \preccurlyeq)$ is stable.
\end{abstract}


\title[On the problem of stability of AECs of modules]{On the problem of stability of abstract elementary classes of modules}


\thanks{No. 1271 on Shelah's publication list. Research of G. Paolini was  supported by project PRIN 2022 ``Models, sets and classifications", prot. 2022TECZJA, and by INdAM Project 2024 (Consolidator grant) ``Groups, Crystals and Classifications''. Research of S. Shelah was supported by SF 2320/23: The Israel Science Foundation (ISF) (2023-2027).}

\author{Gianluca Paolini}
\address{Department of Mathematics ``Giuseppe Peano'', University of Torino, Via Carlo Alberto 10, 10123, Italy.}
\email{gianluca.paolini@unito.it}

\author{Saharon Shelah}
\address{Einstein Institute of Mathematics,  The Hebrew University of Jerusalem, Israel \and Department of Mathematics,  Rutgers University, U.S.A.}
\email{shelah@math.huji.ac.il}

\date{\today}
\maketitle

\setcounter{tocdepth}{1}



\section{Introduction}

Following the development of classification theory for first-order logic \cite{shelah_class}, the second-named author initiated a program aimed at developing an abstract framework for model theory and classification theory. This led to the area of model theory known as Abstract Elementary Classes (AECs) \cite{shelah_abstr_ele_cla}. A longstanding challenge in this theory has been the scarcity of new interesting examples beyond those arising from first-order logic. A significant breakthrough came through Zilber's work on complex exponentiation (cf. \cite{zilber}). More recently, largely due to the work of Mazari-Armida (see e.g. \cite{armida_fuchs, armida2, armida3, armida}), the model theory community has recognized that module theory provides a rich source of applications for the general theory of AECs (see also the recent survey \cite{boney}). A central open problem in this area of model theory is the following question, formulated by Mazari-Armida in \cite{armida_fuchs}.

\begin{question}\label{the_question} Let $R$ be a ring and let $\leq_{\mrm{pure}}$ denote the pure submodule relation. If $(\mathbf{K}, \leq_{\mrm{pure}})$ is an abstract elementary class with $\mathbf{K} \subseteq R$-$\mrm{Mod}$, is $(\mathbf{K}, \leq_{\mrm{pure}})$ stable? Is this true when $R = \mathbb{Z}$? Under what conditions on $R$ does this hold?
\end{question}

We note that this question is inspired by a classical result from model theory: for any ring $R$, every complete first-order theory of $R$-modules is stable. This result, together with the well-known elimination of quantifiers down to $\mrm{pp}$-formulas, makes the first-order model theory of modules one of the most well-behaved areas of application of model theory to algebraic structures. For extensive background on the first-order model theory of modules see e.g. the classical references \cite{prest1, prest2}.
	
\medskip 	In this paper we answer Question~\ref{the_question} in the negative, more precisely, we prove:

\begin{theorem}\label{counterexample_th} There is a class of torsion-free abelian groups $\mathbf{K}$ such \mbox{that $(\mathbf{K}, \leq_{\mrm{pure}})$} is an $\mrm{AEC}$ and $(\mathbf{K}, \leq_{\mrm{p}})$ is unstable, where $\leq_{\mrm{pure}}$ denotes the pure subgroup relation.
	\end{theorem}
	
	At first glance, our results might seem discouraging but this is only half of the story. In fact, we will see that despite the failure of stability, $\mrm{AEC}$s of $R$-modules are still as well-behaved as possible from the point of view of classification theory, in an appropriate sense. We first recall the notion of $\kappa$-locality, a.k.a. $\kappa$-tameness, \cite{394, gross, 932} (notice that most references refer to this notion as $\kappa$-tameness): an $\mrm{AEC}$ $(\mathbf{K}, \preccurlyeq)$ is said to be $\kappa$-local (a.k.a. $\kappa$-tame) if whenever two types from $(\mathbf{K}, \preccurlyeq)$ differ, they already differ on a model of size $\leq \kappa$ (cf. \ref{def_tame}).  In recent years, the notion of $\kappa$-locality (a.k.a. $\kappa$-tameness) has been recognized as central in the study of $\mrm{AEC}$s, and this additional assumption is \mbox{often made in case studies (cf. \cite{tame_surv}).}
	
\smallskip

Now, despite the unstability from Theorem~\ref{counterexample_th}, under the assumption of $\kappa$-locality for some $\kappa \geq \mrm{LS}(\mathfrak{K})$, we establish the next best form of stability possible, namely what we call {\em almost stability} (cf. \ref{def_stability}(2)(3)). This is a notion that we introduce in this paper, which is equivalent to the usual notion of stability under the assumption of amalgamation. In brief, almost stability means that there are only a few orbital types over $M$, once we restrict to a specific strong extension $N$ of $M$ (that is why stability and almost-stability coincide under the assumption of amalgamation). We believe that this notion is of independent interest and we hope that it will inspire future studies and new directions in the theory of $\mrm{AEC}$s.
	
%
	
	\begin{theorem}\label{stability_th} Let $\mathfrak{K}$ be an $\mrm{AEC}$ of $R$-modules s.t. $\mathfrak{K}$ is $\kappa$-local for some $\kappa \geq \mrm{LS}(\mathfrak{K})$. 
	\begin{enumerate}[(1)]
	\item There is $\xi > \kappa$ such that, for every cardinal $\mu$ satisfying
	$$\mu = \mu^{< \xi} + \sum \{2^{2^{\sigma}} : \sigma < \xi\},$$
we have that $\mathfrak{K}$ is {\em almost} $\mu$-stable (cf. \ref{def_stability}(2)(3)).
	\item If in addition $\mathfrak{K}$ has amalgamation and $\mu$ is as in (1), then $\mathfrak{K}$ is $\mu$-stable.
\end{enumerate}
	\end{theorem}
	
	We note that an analogous transfer of stability was obtained by Baldwin, Kueker, and VanDieren \cite{BKV06}, who showed that $\omega$-stability implies stability in all cardinals under the assumptions of tameness and locality (in the sense of \cite{BS08}); our Theorem~1.3 obtains a similar conclusion for AECs of modules via a different route, replacing those abstract hypotheses with the syntactic machinery of \cite{977}, where quantifier elimination is extended from first-order positive primitive formulas to positive existential formulas in $\mathfrak{L}_{\infty, \theta}$. Theorem~\ref{stability_th} relies crucially on the results from \cite{977}. Notice that an updated version of \cite{977} can be found on Shelah's archive.


\smallskip By known consistency results on locality (see e.g. \cite{boney_tame}), we deduce:
	
	\begin{corollary} If there is a strongly compact cardinal $\kappa$ and $\mathfrak{K}$ is an $\mrm{AEC}$ of $R$-modules with amalgamation such that $\kappa > \mrm{LS}(\mathfrak{K})$, then $\mathfrak{K}$ is stable.
\end{corollary}

	Notice that $\mrm{AEC}$s of $R$-modules which arise from first-order theories are local and have amalgamation and so our theorem can be seen as the most general form of stability for $R$-modules currently known in the literature. Explicitly, we deduce:

	\begin{corollary} Let $\mathfrak{K} = (\mathbf{K}, \preccurlyeq)$ be such that $\mathbf{K}$ is a complete first-order theory of $R$-modules and $\preccurlyeq$ is the relation of elementary first-order substructure. Then $(\mathbf{K}, \preccurlyeq)$ is stable in the sense of first-order logic.
\end{corollary}

Corollary~1.5 can be seen as a generalization, to the setting of abstract elementary classes, of the classical result independently established by Baur \cite[Theorem~1]{baur} and Fisher \cite{fisher} in 1975, which asserts that any complete first-order theory of $R$-modules is stable, for any ring $R$. Their proof relies in an essential way on the syntactic form of first-order formulas in the language of
$R$-modules; in contrast, the AEC framework considered here, following the groundbreaking work of Mazari-Armida \cite{armida_fuchs, armida2, armida3, armida}, encompasses a considerably larger class of modules, where such syntactic control is no longer available. 
	

	The challenge that we leave for future studies are the following questions.

\begin{question}\label{open_question}
 \begin{enumerate}[(1)]
\item Is there (consistently) a ring $R$ and an $\mrm{AEC}$ of $R$-modules $\mathfrak{K}$ which is not $\kappa$-local for unboundedly many $\kappa$ among the cardinals $\theta$ such that $\theta$ is below the first strongly compact cardinal $> \mrm{LS}(\mathfrak{K})$.
\item Is there (consistently) a ring $R$ and an $\mrm{AEC}$ of $R$-modules $\mathfrak{K}$ which is not $\kappa$-local for unboundedly many $\kappa$ among the cardinals $\theta$ such that $\theta$ is below $\beth_{(2^{\mrm{\mrm{LS}(\mathfrak{K}}})^+}$?
\end{enumerate}
\end{question}

	We conjecture that the answer to \ref{open_question}(2) is yes.

\section{Preliminaries}\label{prel_sec}

	\begin{notation} Given a formula $\varphi$ and ordinals $\alpha, \beta, \gamma$, when we write $\varphi(\bar{x}_\alpha, \bar{y}_\beta, \bar{z}_\gamma)$ we mean that $\bar{x}_\gamma = (x_i : i < \alpha)$,  $\bar{y}_\beta = (y_i : i < \beta)$ and $\bar{z}_\gamma = (z_i : i < \gamma)$. 
\end{notation}

	Concerning the following crucial definition see e.g. \cite[Definition~8.7]{Bal_categoricity}.

	\begin{definition}\label{notation_types} Let $\mathfrak{K} = (\mathbf{K}_{\mathfrak{K}}, \preccurlyeq_{\mathfrak{K}}) = (\mathbf{K}, \preccurlyeq)$ be an $\mrm{AEC}$. Given $(\bar{b}, A, N)$, where $N \in \mathbf{K}$, $A \subseteq N$, and $\bar{b}$ is a sequence in $N$, the orbital type (a.k.a. the Galois type) of $\bar{b}$ over $A$ in $N$, denoted by $\gtp_{\mathfrak{K}}(\bar{b}/A; N)$, is the equivalence class of $(\bar{b}, A, N)$ modulo $E_{\mathfrak{K}}$, where $E_{\mathfrak{K}}$ is the transitive closure of the relation $E_{\mathfrak{K}}^{\mrm{at}}$, where, for $N_1, N_2 \in \mathbf{K}$, $\bar{b}_\ell \in N_\ell$ and $A \subseteq N_1, N_2$, we let $(\bar{b}_1, A_1, N_1)E_{\mathfrak{K}}^{\mrm{at}} (\bar{b}_2, A_2, N_2)$ if $A := A_1 = A_2$, and there exist $N_* \in \mathbf{K}$ and $\mathfrak{K}$-embeddings $f_\ell : N_\ell \rightarrow_A N_*$ for $\ell \in \{1, 2\}$ such that $f_1(\bar{b}_1) = f_2(\bar{b}_2)$.
If $M \in {\mathbf{K}}$ and $\gamma$ is an ordinal, we let $\gS^\gamma_{\mathfrak{K}}(M)$ to be the set:
$$\{\gtp_{\mathfrak{K}}(\bar{b}/M; N) : M \sprec N \in K \text{ and } \bar{b} \in N^\gamma\}.$$
When $\gamma = 1$, we may write $\gS_{\mathfrak{K}}(M)$ instead of $\gS^1_{\mathfrak{K}}(M)$. Finally, we let 
$$\gS^{<\infty}_{\mathfrak{K}}(M) = \bigcup_{\gamma \in \mathrm{OR}} \gS^\gamma_{\mathfrak{K}}(M).$$
\end{definition}

\begin{notation} Let $\mathfrak{K} = (\mathbf{K}_{\mathfrak{K}}, \preccurlyeq_{\mathfrak{K}}) = (\mathbf{K}, \preccurlyeq)$ be an $\mrm{AEC}$. For $\lambda \in \mrm{Card}$, we let 
$$\mathbf{K}_\lambda = \{M \in \mathbf{K} : |M| = \lambda \}.$$
\end{notation}

\begin{definition}\label{def_stability}
Let $\mathfrak{K} = (\mathbf{K}_{\mathfrak{K}}, \preccurlyeq_{\mathfrak{K}}) = (\mathbf{K}, \preccurlyeq)$ be an $\mrm{AEC}$, $\lambda \in \mrm{Card}$ and $\gamma \in \mrm{Ord}$. 
\begin{enumerate}[(1)]
\item We say that $\mathfrak{K}$ is $(\lambda, \gamma)$-stable if for any $M \in \mathbf{K}_\lambda$ we have that $|\gS^\gamma_{\mathfrak{K}}(M)| \leq \lambda$. 
\item We say that $\mathfrak{K}$ is almost $(\lambda, \gamma)$-stable if for any $M, N \in \mathbf{K}$ with $M \sprec N$ and $M \in \mathbf{K}_\lambda$ we have that $|\gS^\gamma_{\mathfrak{K}}(M; N)| \leq \lambda$, where
$$\gS^\gamma_{\mathfrak{K}}(M; N) = \{\gtp(\bar{c}/M, N) : \bar{c} \in N^\gamma\}.$$
\item If $\gamma = 1$, then we simply say (almost) $\lambda$-stable and we write $\gS_{\mathfrak{K}}(M; N)$.
\item $\mathfrak{K}$ is (almost) stable if it is (almost) $\mu$-stable for unboundedly many $\mu \in \mrm{Card}$.
\end{enumerate} 
\end{definition}

	\begin{remark} Notice that in some references (e.g. the recent survey \cite{boney}), stability is defined as follows: $\mathfrak{K}$ is stable if it is $\mu$-stable for some $\mu \in \mrm{Card}$. 
\end{remark}

	\begin{observation} Notice that if $\mathfrak{K}$ has the amalgamation property, then $\mathfrak{K}$ is almost $(\lambda, \gamma)$-stable if and only if it is $(\lambda, \gamma)$-stable. Furthermore, recalling the definition of $E_{\mathfrak{K}}$ and $E^{\mrm{at}}_{\mathfrak{K}}$ from Notation~\ref{notation_types}, we have that $E_{\mathfrak{K}} = E'_{\mathfrak{K}}$. On the other hand, without a global amalgamation hypotheses an extension
$N$ over $M$ can be separately amalgamated over $M$ with two incompatible extensions. Thus, without assuming amalgamation almost $(\lambda, \gamma)$-stability does not imply $(\lambda, \gamma)$-stability.
\end{observation}

\begin{definition}\label{def_tame}
Let $\mathfrak{K} = (\mathbf{K}_{\mathfrak{K}}, \preccurlyeq_{\mathfrak{K}}) = (\mathbf{K}, \preccurlyeq)$ be an $\mrm{AEC}$ and $\kappa$ an infinite cardinal.
\begin{enumerate}[(1)]
\item We say that $\mathfrak{K}$ is $({<}\kappa)$-$\gamma$-local (a.k.a. $({<}\kappa)$-$\gamma$-tame) if for any $M \in \mathbf{K}$ and $p \neq q \in \gS^\gamma_{\mathfrak{K}}(M)$, there exists $M_0 \subseteq M$ such that $|M_0| < \kappa$ and $p \upharpoonright M_0 \neq q \upharpoonright M_0$.
\item When we say that $\mathfrak{K}$ is $\kappa$-$\gamma$-local we mean that $\mathfrak{K}$ is $({<}\kappa^+)$-$\gamma$-local.
\item By $({<}\kappa)$-local we mean $({<}\kappa)$-$1$-local and similarly by $\kappa$-local we mean $\kappa$-$1$-local.
\end{enumerate}
\end{definition}

	\begin{remark} The notion of $\kappa$-locality was used in \cite{394} under the assumption of the amalgamation property, and  in \cite{932} without this assumption. We remark that our notion of $\kappa$-locality corresponds to what is usually called $\kappa$-tameness. Notice that in \cite{BS08}  the term
$\kappa$-locality is reserved for a distinct condition on chains of models.
\end{remark}

Notice that that the syntactic order property presented in \ref{def_order_prop}(2) has already been considered in the literature, see e.g. \cite[2.19]{vasey_2} and references therein.

\begin{definition}\label{def_order_prop}
Let $\mathfrak{K} = (\mathbf{K}_{\mathfrak{K}}, \preccurlyeq_{\mathfrak{K}}) = (\mathbf{K}, \preccurlyeq)$ be an $\mrm{AEC}$, $\lambda \in \mrm{Card}$ and $\gamma \in \mrm{Ord}$. 
\begin{enumerate}[(1)]
\item We say that $\mathfrak{K}$ has the $(\lambda, \gamma)$-order property if there are $M \in \mathbf{K}$ and $(\bar{a}_i : i < \lambda)$ inside $M$ with $\mrm{lg}(\bar{a}_i) = \gamma$, for all $i < \lambda$, such that for any $i_0 < j_0 < \lambda$ and $i_1 < j_1 < \lambda$, $\gtp_{{\mathfrak{K}}}(\bar{a}_{i_0} \bar{a}_{j_0} / \emptyset; N) \neq \gtp_{{\mathfrak{K}}}(\bar{a}_{j_1} \bar{a}_{i_1} / \emptyset; N)$.
\item We say that $\mathfrak{K}$ has the syntactic $(\lambda, \kappa, \gamma, \Delta)$-order property if there are $M \in \mathbf{K}$ and $(\bar{a}_i : i < \lambda)$ inside $M$ with $\mrm{lg}(\bar{a}_i) = \gamma$, for all $i < \lambda$, and contradictory $\varphi_1(\bar{x}_\gamma, \bar{y}_\gamma),  \varphi_2(\bar{x}_\gamma, \bar{y}_\gamma) \in \Delta \subseteq \mathfrak{L}_{\infty, \kappa^+}(\tau_{\mathfrak{K}})$ (e.g. $\varphi_1(\bar{x}_\gamma, \bar{y}_\gamma)$ is $\neg \varphi_2(\bar{x}_\gamma, \bar{y}_\gamma)$) s.t.:
$$i < j < \lambda \Rightarrow M \models \varphi_1(\bar{a}_i, \bar{a}_j) \wedge \varphi_2(\bar{a}_j, \bar{a}_i).$$
	\item If $\kappa = \mrm{LS}(\mathfrak{K})$ we simply say syntactic $(\lambda, \gamma, \Delta)$-order property. Furthermore, if $\gamma = 1$, then we simply say (syntactic) $\lambda$-order ($(\lambda, \Delta)$-order) property.
\item We say that $\mathfrak{K}$ does not have the $\gamma$-order property (resp. syntactic $(\lambda, \Delta)$-order property) if it does not have the $(\mu, \gamma)$-order property (resp. syntactic $(\mu, \gamma, \Delta)$-order property) for some $\mu \in \mrm{Card}$.
\item We define almost $({<}\kappa)$-local and almost $\kappa$-local similarly.
\end{enumerate}
\end{definition}



\begin{fact}[{\cite[Theorem~1.3]{boney_tame}}]\label{fact_Boney} If $\mathfrak{K}$ is an $\mrm{AEC}$ with $\mrm{LS}(\mathfrak{K}) < \kappa$ and $\kappa$ is strongly compact, then
$\mathfrak{K}$ is $\kappa$-local.
\end{fact}

	\begin{definition}\label{def_syntactic_stable} Let $\mathfrak{K} = (\mathbf{K}_{\mathfrak{K}}, \preccurlyeq_{\mathfrak{K}}) = (\mathbf{K}, \preccurlyeq)$ be an $\mrm{AEC}$,  $\Delta \subseteq \mathfrak{L}_{\infty, \kappa^+}(\tau_{\mathfrak{K}})$ and $\gamma < \kappa^+$.
	\begin{enumerate}[(1)]
	\item We say that $\mathfrak{K}$ is syntactically $(\lambda, \kappa, \gamma, \Delta)$-stable \underline{when} for every $M \in \mathbf{K}_\lambda$ we have that $|\mathbf{S}^\gamma_{(\Delta, \mathfrak{K})}(M)| \leq \lambda$, where:
	$$\mathbf{S}^\gamma_{(\Delta, \mathfrak{K})}(M) = \{\mrm{tp}_\Delta(\bar{c}/M; N)\} : N  \in \mathbf{K}, \; M \sprec N, \; \bar{c} \in N^\gamma \}.$$
	\item We say that $\mathfrak{K}$ is syntactically almost $(\lambda, \kappa, \gamma, \Delta)$-stable \underline{when} for every $M, N \in \mathbf{K}_\lambda$ with $M \sprec N$ and $M \in \mathbf{K}_\lambda$ we have that $|\mathbf{S}^\gamma_{(\Delta, \mathfrak{K})}(M; N)| \leq \lambda$, where:
	$$\mathbf{S}^\gamma_{(\Delta, \mathfrak{K})}(M; N) = \{\mrm{tp}_\Delta(\bar{c}/M; N)\} : \bar{c} \in N^\gamma \}.$$
	\item If $\kappa$ is minimal such that $\kappa \geq \mrm{LS}(\mathfrak{K})$ and $\Delta \subseteq \mathfrak{L}_{\infty, \kappa^+}(\tau_{\mathfrak{K}})$, then we may omit $\kappa$.
	\item  If $\Delta = \mathfrak{L}_{\infty, \kappa^+}(\tau_{\mathfrak{K}})$, then we simply say syntactically (almost) $(\lambda, \kappa, \gamma)$-stable.
\end{enumerate}
\end{definition}


	We need the following crucial fact from \cite{1184}. Notice that although the following fact is not explicitly stated in journal version of \cite{1184}, it follows from the proof of the second main theorem (the ``Tarski-Vaught'' criterion for $\mrm{AEC}$s); furthermore, this fact is explicitly stated in the latest arXiv version of the paper \cite[Claim~3.1]{1184}.

\begin{fact}[{\cite{1184}}]\label{Villaveces_fact} Let $\mathfrak{K} = (\mathbf{K}_{\mathfrak{K}}, \preccurlyeq_{\mathfrak{K}}) = (\mathbf{K}, \preccurlyeq)$ be an $\mrm{AEC}$ and let $(\lambda_{\kappa_0}, \kappa_0)$ be as in \cite{1184}, i.e., $\kappa_0 = \mrm{LS}(\mathfrak{K}) + |\tau_{\mathfrak{K}}|$ and $\lambda_{\kappa_0} = \beth_2(\kappa_0)^{++}$. More generally, for $\kappa \geq \mrm{LS}(\mathfrak{K})+ |\tau_{\mathfrak{K}}|$, let $\lambda_{\kappa} = \beth_2(\kappa)^{++}$. 
	Then there is $\varphi^\star_{\kappa}(\bar{x}_{\kappa}) \in \mathfrak{L}_{\lambda^+, \kappa^+}(\tau_{\mathfrak{K}})$ such that:
	\begin{enumerate}[$(\star)$]
	\item if $N \in \mathbf{K}$, $\bar{a} \in N^\kappa$ and $N \restriction \bar{a}$ is a substructure of $N$, then we have:
	$$N \restriction \bar{a} \preccurlyeq N \;\; \Leftrightarrow  \;\; N \models \varphi^\star_{\kappa}(\bar{a}).$$
	\end{enumerate}
\end{fact}

\section{Almost stability for AECs of $R$-modules}

The aim of this section is to prove Theorem~\ref{stability_th}.

\begin{hypothesis}\label{hyp_general_sec}
	\begin{enumerate}[(1)]
	\item $\mathfrak{K} = (\mathbf{K}_{\mathfrak{K}}, \preccurlyeq_{\mathfrak{K}}) = (\mathbf{K}, \preccurlyeq)$ is a fixed $\mrm{AEC}$. 
	\item $\kappa \geq \mrm{LS}(\mathfrak{K}) + |\tau_{\mathfrak{K}}|$ and $\lambda = \beth_2(\kappa)^{++}$.
	\item $\gamma \leq \kappa$ is an ordinal.
\end{enumerate}
\end{hypothesis}

	\begin{notation}\label{notation_W}
	\begin{enumerate}[(1)]
	\item Let $\mathcal{W}^{\mrm{small}} = \mathcal{W}^{\mrm{small}}_{(\mathfrak{K}, \kappa, \gamma)}$ be the class of quintuples of the form
	$$\mathfrak{u} = (M_\mathfrak{u}, N_\mathfrak{u}, \bar{a}_\mathfrak{u}, \bar{b}_\mathfrak{u}, \bar{c}_\mathfrak{u}) = (M, N, \bar{a}, \bar{b}, \bar{c})$$ such that:
	\begin{enumerate}[(a)]
	\item $M \sprec N$, $|M| \leq |N| \leq \kappa$ (we write ``small'' since we ask $|N| \leq \kappa$ here);
	\item $\bar{a}$ lists $M$ and $\bar{b}$ lists $N$;
	\item $\bar{c} \in N^\gamma$;
	\end{enumerate}
	\item Let $\mathcal{W}^{\mrm{large}} = \mathcal{W}^{\mrm{large}}_{(\mathfrak{K}, \kappa, \gamma)}$ be the class of quadruples of the form 
	$$\mathfrak{u} = (M_\mathfrak{u}, N_\mathfrak{u}, \bar{a}_\mathfrak{u}, \bar{c}_\mathfrak{u}) = (M, N, \bar{a}, \bar{c})$$
such that:
	\begin{enumerate}[(a)]
	\item $M \sprec N$, $|M| \leq \kappa$;
	\item $\bar{a}$ lists $M$;
	\item $\bar{c} \in N^\gamma$.
	\end{enumerate}
	\end{enumerate}
\end{notation}

	\begin{mclaim}\label{main_abstract_claim} In the context of \ref{notation_W}.
	\begin{enumerate}[(A)]
	\item For $\mathfrak{w} = (M_\mathfrak{u}, N_\mathfrak{u}, \bar{a}_\mathfrak{u}, \bar{b}_\mathfrak{u}, \bar{c}_\mathfrak{u}) \in \mathcal{W}^{\mrm{small}}_{(\mathfrak{K}, \kappa, \gamma)}$, there is 
	$$\theta(\bar{z}, \bar{y}_\kappa, \bar{x}_\kappa) = \theta_{\mathfrak{w}}(\bar{z}, \bar{y}_\kappa, \bar{x}_\kappa) \in \mathfrak{L}_{\lambda^+, \kappa^+}(\tau_{\mathfrak{K}})$$ such that:
	\begin{enumerate}[(a)]
	\item if $\mathfrak{w}_1 \cong \mathfrak{w}_2 \in \mathcal{W}^{\mrm{small}}_{(\mathfrak{K}, \kappa, \gamma)}$ are isomorphic (where this means what you expect), then $\theta_{\mathfrak{w}_1}(\bar{z}_\gamma, \bar{y}_\kappa, \bar{x}_\kappa) = \theta_{\mathfrak{w}_2}(\bar{z}_\gamma, \bar{y}_\kappa, \bar{x}_\kappa)$;
	\item if $\mathfrak{w} \in \mathcal{W}^{\mrm{small}}_{(\mathfrak{K}, \kappa, \gamma)}$, \underline{then} $N_{\mathfrak{w}} \models \theta_{\mathfrak{w}}(\bar{c}_{\mathfrak{w}}, \bar{b}_{\mathfrak{w}}, \bar{a}_{\mathfrak{w}})$;
	\item if $\mathfrak{w}_1, \mathfrak{w}_2 \in \mathcal{W}^{\mrm{small}}_{(\mathfrak{K}, \kappa, \gamma)}$ and $\bar{a}_{\mathfrak{w}_1} = \bar{a}_{\mathfrak{w}_2}$ (so $M_{\mathfrak{w}_1} = M_{\mathfrak{w}_2}$), \underline{then}
	 $$\gtp(\bar{c}_{\mathfrak{w}_1}/M_{\mathfrak{w}_1}; N_{\mathfrak{w}_1}) = \gtp(\bar{c}_{\mathfrak{w}_2}/M_{\mathfrak{w}_2}; N_{\mathfrak{w}_2}) \Leftrightarrow \theta_{\mathfrak{w}_1} = \theta_{\mathfrak{w}_2}.$$
	\end{enumerate}
	\item For $\mathfrak{w} = (M_\mathfrak{u}, N_\mathfrak{u}, \bar{a}_\mathfrak{u}, \bar{c}_\mathfrak{u}) \in \mathcal{W}^{\mrm{large}}_{(\mathfrak{K}, \kappa, \gamma)}$, there is 
	$$\psi(\bar{z}_\gamma, \bar{x}_\kappa) = \psi_{\mathfrak{w}}(\bar{z}_\gamma, \bar{x}_\kappa) \in \mathfrak{L}_{\lambda^+, \kappa^+}(\tau_{\mathfrak{K}})$$
	 such that:
	\begin{enumerate}[(a)]
	\item if $\mathfrak{w}_1 \cong \mathfrak{w}_2 \in \mathcal{W}^{\mrm{large}}_{(\mathfrak{K}, \kappa, \gamma)}$ are isomorphic (where this means what you expect), then $\psi_{\mathfrak{w}_1}(\bar{z}_\gamma, \bar{x}_\kappa) = \psi_{\mathfrak{w}_2}(\bar{z}_\gamma, \bar{x}_\kappa)$;
	\item if $\mathfrak{w} \in \mathcal{W}^{\mrm{large}}_{(\mathfrak{K}, \kappa, \gamma)}$, \underline{then} $N_{\mathfrak{w}} \models \psi_{\mathfrak{w}}(\bar{c}_{\mathfrak{w}}, \bar{a}_{\mathfrak{w}})$;
	\item if $\mathfrak{w}_1, \mathfrak{w}_2 \in \mathcal{W}^{\mrm{large}}_{(\mathfrak{K}, \kappa, \gamma)}$ and $M_{\mathfrak{w}_1} = M_{\mathfrak{w}_2}$, \underline{then}
	 $$\gtp(\bar{c}_{\mathfrak{w}_1}/M_{\mathfrak{w}_1}; N_{\mathfrak{w}_1}) = \gtp(\bar{c}_{\mathfrak{w}_2}/M_{\mathfrak{w}_2}; N_{\mathfrak{w}_2}) \Leftrightarrow \psi_{\mathfrak{w}_1} = \psi_{\mathfrak{w}_2}.$$
	\end{enumerate}
	\end{enumerate}
\end{mclaim}

	\begin{proof} We prove clause (A). Given $\mathfrak{m} = (M_\mathfrak{m}, N_\mathfrak{m}, \bar{a}_\mathfrak{m}, \bar{b}_\mathfrak{m}, \bar{c}_\mathfrak{m}) \in \mathcal{W}^{\mrm{small}}_{(\mathfrak{K}, \kappa, \gamma)}$, Let $\theta^0_{\mathfrak{w}}$ be the conjunction of formulas $\varphi(\bar{z}_\gamma \restriction u_3, \bar{y}_\kappa \restriction u_2, \bar{x}_\kappa \restriction u_1)$, where $u_1, u_2$ are finite subsets of $\kappa$, $u_3$ is a finite subset of $\gamma$, $\varphi$ is an atomic formula or the negation of an atomic formula and $N_{\mathfrak{m}} \models \varphi(\bar{c}_\gamma \restriction u_3, \bar{b}_\kappa \restriction u_2, \bar{a}_\kappa \restriction u_1)$. Now, $\theta^0_{\mathfrak{m}}$ satisfies clauses (A)(a)(b) but not necessarily clause (A)(c). We define an equivalence relation $E^{\mrm{small}}_{(\mathfrak{K}, \kappa, \gamma)}$ on $\mathcal{W}^{\mrm{small}}_{(\mathfrak{K}, \kappa, \gamma)}$ by requiring that $\mathfrak{m}_1 E^{\mrm{small}}_{(\mathfrak{K}, \kappa, \gamma)} \mathfrak{m}_2$ if and only there is a mapping $\pi$ such that (for the definition of the image of a type see e.g. \cite[Definition~7.8]{vasey_lec_notes}):
\begin{enumerate}[$(\cdot_1)$]
	\item $\pi(a_{\mathfrak{m}_1, i}) =a_{\mathfrak{m}_2, i}$ is an isomorphism from $M_{\mathfrak{m}_1}$ onto $M_{\mathfrak{m}_2}$ such that $\bar{a}_{\mathfrak{m}_1} \mapsto \bar{a}_{\mathfrak{m}_2}$;
	\item $\gtp(\bar{c}_{\mathfrak{m}_2}/M_{\mathfrak{m}_2}; N_{\mathfrak{m}_2})= \pi(\gtp(\bar{c}_{\mathfrak{m}_1}/M_{\mathfrak{m}_1}; N_{\mathfrak{m}_1}))$.
\end{enumerate}
Notice that for $\mathfrak{m} \in \mathcal{W}^{\mrm{small}}_{(\mathfrak{K}, \kappa, \gamma)}$ the family of formulas
$$\Phi_{\mathfrak{m}} = \{\theta^0_{\mathfrak{m}_1} : \mathfrak{m}_1 E^{\mrm{small}}_{(\mathfrak{K}, \kappa, \gamma)} \mathfrak{m} \}$$
is a set with $\leq 2^\kappa$ members. Lastly, the formula $\theta_{\mathfrak{m}} = \bigvee \Phi_{\mathfrak{m}}$ is as required. 

\smallskip \noindent
We prove clause (B). For every $\mathfrak{w} \in \mathcal{W}^{\mrm{large}}_{(\mathfrak{K}, \kappa, \gamma)}$ we define $\mrm{nb}(\mathfrak{w})$ as follows (where ``$\mrm{nb}$'' stands for ``neighborhood''):
\begin{enumerate}[$(\cdot)$]
	\item $\mathfrak{u} \in \mrm{nb}(\mathfrak{w})$ iff $\mathfrak{u} \in \mathcal{W}^{\mrm{small}}_{(\mathfrak{K}, \kappa, \gamma)}$,  $M_{\mathfrak{u}} = M_{\mathfrak{w}}$, $\bar{a}_{\mathfrak{u}} = \bar{a}_{\mathfrak{w}}$, $\bar{c}_{\mathfrak{u}} = \bar{c}_{\mathfrak{w}}$, $N_\mathfrak{u} \preccurlyeq_{\mathfrak{K}} N_\mathfrak{m}$ and $|N_\mathfrak{u}| \leq \kappa$, so in particular we have that $\mathfrak{u} = (M_\mathfrak{u}, N_\mathfrak{u}, \bar{a}_\mathfrak{u}, \bar{b}_\mathfrak{u}, \bar{c}_\mathfrak{u}) \in \mathcal{W}^{\mrm{small}}_{(\mathfrak{K}, \kappa, \gamma)}$.
\end{enumerate}
Finally, recalling the formula $\varphi^\star_{\kappa}(\bar{x}_{\kappa})$ from \ref{Villaveces_fact}, and noticing that by assumption we have that $\kappa \geq \kappa_0 = \mrm{LS}(\mathfrak{K}) + |\tau_{\mathfrak{K}}|$, we define $\psi_{\mathfrak{w}}(\bar{z}_\gamma, \bar{x}_\kappa)$ as the following formula:
$$\bigvee \{\exists \bar{y}_\kappa (\varphi^\star_{\kappa}(\bar{y}_{\kappa}) \wedge \theta_{\mathfrak{u}}(\bar{z}_\gamma, \bar{y}_\kappa, \bar{x}_\kappa))  : \mathfrak{u} \in \mrm{nb}_{\mathfrak{w}}\},$$
where $\theta_{\mathfrak{u}}(\bar{z}_\gamma, \bar{y}_\kappa, \bar{x}_\kappa)$ is as in clause (B) of this claim. Then $\psi_{\mathfrak{w}}(\bar{z}_\gamma, \bar{x}_\kappa)$ is as desired. 
\end{proof}

	\begin{claim}\label{equivalence_syntactic} Suppose that $\kappa \geq \mrm{LS}(\mathfrak{K}) + |\gamma|$ and that $\mathfrak{K}$ is $\kappa$-$\gamma$-local and let
	$$\Delta = \Delta_{(\mathfrak{K}, \kappa, \gamma)} = \{\psi_{\mathfrak{w}} : \mathfrak{w} \in \mathcal{W}^{\mrm{large}}_{(\mathfrak{K}, \kappa, \gamma)}\},$$
where $\psi_{\mathfrak{w}}$ is as in \ref{main_abstract_claim}(B). \underline{Then}, for every $\mu \in \mrm{Card}$, the following are equivalent:
\begin{enumerate}[(1)]
\item $\mathfrak{K}$ is almost $(\mu, \gamma)$-stable;
\item $\mathfrak{K}$ is syntactically almost $(\mu, \gamma, \Delta)$-stable.
\end{enumerate}
\end{claim}

	\begin{proof}
Assume that $M \in \mathbf{K}_\mu$ and $M \preccurlyeq_{\mathfrak{K}} N$. 
\begin{enumerate}[$(*_1)$]
	\item It suffices to prove $(a) \Leftrightarrow (b)$, where:
	\begin{enumerate}[$(a)$]
	\item $\{\gtp(\bar{b}/M; N) : \bar{b} \in N^\gamma \}$ has cardinality $\leq \mu$;
	\item $\{\mrm{tp}_{\Delta}(\bar{b}/M; N); \bar{b} \in N^\gamma \}$ has cardinality $\leq \mu$.
	\end{enumerate}
\end{enumerate}
In fact we shall prove more. First we observe the following.
\begin{enumerate}[$(*_2)$]
	\item It suffices to prove that, for $\bar{c}_1, \bar{c}_2 \in N^\gamma$,  $(c)_{(\bar{c}_1, \bar{c}_2)} \Leftrightarrow (d)_{(\bar{c}_1, \bar{c}_2)}$, where:
	\begin{enumerate}[$(c)_{(\bar{c}_1, \bar{c}_2)}$]
	\item $\gtp(\bar{c}_1/M; N) = \gtp(\bar{c}_2/M; N)$;
		\end{enumerate}
	\begin{enumerate}[$(d)_{(\bar{c}_1, \bar{c}_2)}$]
	\item $\mrm{tp}_\Delta(\bar{c}_1/M;N) = \mrm{tp}_\Delta(\bar{c}_2/M;N)$.
	\end{enumerate}
	\end{enumerate}
	So we proceed to the proof that for $\bar{c}_1, \bar{c}_2 \in N^\gamma$ we have that  $(c)_{(\bar{c}_1, \bar{c}_2)} \Leftrightarrow (d)_{(\bar{c}_1, \bar{c}_2)}$. To prove the ``left-to-right'' implication, first of all observe that all the formulas in $\Delta$ are formulas in the logic $\mathfrak{L}_{\lambda^+, \kappa^+}(\tau_{\mathfrak{K}})$ and so it suffices to show that for every $M' \preccurlyeq_{\mathfrak{K}} M$ with $|M'| \leq \kappa$ we have that:
\begin{enumerate}[$(d')_{(\bar{c}_1, \bar{c}_2)}$]
	\item $\mrm{tp}_\Delta(\bar{c}_1/M'; N) = \mrm{tp}_\Delta(\bar{c}_2/M'; N)$.
		\end{enumerate}
Now, to show $(d')_{(\bar{c}_1, \bar{c}_2)}$, it suffices to first define appropriate $\mathfrak{w}_1, \mathfrak{w}_2 \in \mathcal{W}^{\mrm{large}}_{(\mathfrak{K}, \kappa, \gamma)}$ so that $\bar{c}_{\mathfrak{m}_1} = \bar{c}_1$, $\bar{c}_{\mathfrak{m}_2} = \bar{c}_2$, $M'_{\mathfrak{m}_1} = M' = M'_{\mathfrak{m}_2}$, $\bar{a}_{\mathfrak{m}_1} = \bar{a} = \bar{a}_{\mathfrak{m}_2}$ and $N_{\mathfrak{m}_1} = N = N_{\mathfrak{m}_1}$, and second to show that if $(c)_{(\bar{c}_1, \bar{c}_2)}$ holds, then $N \models \psi(\bar{c}_1, \bar{a}) \Leftrightarrow N \models \psi(\bar{c}_2, \bar{a})$, and the latter double implication holds by \ref{main_abstract_claim}(B).
Concerning the ``right-to-left'' implication, since by assumption we have that $\mathfrak{K}$ is $\kappa$-$\gamma$-local it suffices to show that for  every $M' \preccurlyeq_{\mathfrak{K}} M$ with $|M'| = \kappa$ we have that:
\begin{enumerate}[$(c')_{(\bar{c}_1, \bar{c}_2)}$]
	\item $\gtp(\bar{c}_1/M'; N) = \gtp(\bar{c}_2/M'; N)$.
		\end{enumerate}
Let $\bar{a} \in (M')^\kappa$ list $M'$ and let, for $\ell =1, 2$, $\mathfrak{m}_\ell = (M', N, \bar{a}, \bar{c}_\ell)$. Then obviously, for $\ell =1, 2$, $\mathfrak{m}_\ell \in \mathcal{W}^{\mrm{large}}_{(\mathfrak{K}, \kappa, \gamma)}$ and thus using \ref{main_abstract_claim}(B) we conclude.
\end{proof}


	\begin{claim}\label{almost_stability_claim} If the conditions (1)-(5) immediately below are met, \underline{then} $\mathfrak{K}$ is syntactically almost $(\mu, \gamma, \Delta_{(\mathfrak{K}, \kappa, \gamma)})$-stable (recall Definition~\ref{def_syntactic_stable}), where:
	\begin{enumerate}[(1)]
	\item $\mrm{LS}(\mathfrak{K}) \leq \kappa$, $\gamma < \kappa^+$ and $\gamma_* = \gamma + \kappa + \kappa$ (notice that $\gamma_* < \kappa^+$);
	\item $\Delta = \Delta_{(\mathfrak{K}, \kappa, \gamma)} = \{ \psi_{\mathfrak{w}}(\bar{z}_\gamma, \bar{x}_\kappa) : \mathfrak{w} \in \mathcal{W}^{\mrm{large}}_{(\mathfrak{K}, \kappa, \gamma)}\}$;
	\item $\xi$ is such that $\mrm{cf}(\xi) > |\Delta|$;
	\item $\mu = \mu^{< \xi} + \sum \{2^{2^\sigma} : \sigma < \xi\}$;
	\item $\mathfrak{K}$ fails the syntactic $(\xi, \gamma_*, \Delta^+)$-order property, for $\Delta^+$ defined as in $(\boxplus)$ below
		\end{enumerate}
\begin{equation*}
\Delta^+ = \{\delta^\ell_{(\mathfrak{m}_1, \mathfrak{m}_2)} : \ell = 1, 2, \text{ and } \psi_{\mathfrak{m}_1}, \psi_{\mathfrak{m}_2} \in \Delta_{(\mathfrak{K}, \kappa, \gamma)} \text{ are contradictory}\}, \tag*{$(\boxplus)$}
\end{equation*}
where for $\mathfrak{m}_1, \mathfrak{m}_2 \in \Delta_{(\mathfrak{K}, \kappa, \gamma)}$ we define
	\begin{enumerate}[(a)]
	\item $\delta^1_{(\mathfrak{m}_1, \mathfrak{m}_2)}(\bar{y}^1_{\gamma_*}, \bar{y}^2_{\gamma_*})$ is the formula
	$$\psi_{\mathfrak{m}_1}((y^1_\alpha : \alpha < \gamma), (y^2_{\gamma + \alpha} : \alpha < \kappa)) \wedge \psi_{\mathfrak{m}_2}((y^1_\alpha : \alpha < \gamma), (y^2_{\gamma + \kappa + \alpha} : \alpha < \kappa))$$
	\item $\delta^2_{(\mathfrak{m}_1, \mathfrak{m}_2)}(\bar{y}^1_{\gamma_*}, \bar{y}^2_{\gamma_*})$ is the formula
	$$\neg \psi_{\mathfrak{m}_1}((y^1_\alpha : \alpha < \gamma), (y^2_{\gamma + \alpha} : \alpha < \kappa)) \vee \neg \psi_{\mathfrak{m}_2}((y^1_\alpha : \alpha < \gamma), (y^2_{\gamma + \kappa + \alpha} : \alpha < \kappa)).$$
\end{enumerate}
\end{claim}

	\begin{proof} So we are given $M \in \mathbf{K}_\mu$ and $M \preccurlyeq_{\mathfrak{K}} N \in \mathbf{K}$ and we want to prove that
		$$|\{\mrm{tp}_\Delta(\bar{c}_{}/M; N)\} : \bar{c} \in N^\gamma \}| \leq \mu.$$
\begin{enumerate}[$(\star_0)$] 
\item Toward contradiction, for $\alpha < \mu^+$, let $\bar{c}_\alpha \in N^\gamma$ be such that the types $(\mrm{tp}_\Delta(\bar{c}_{\alpha}/M; N) : \alpha < \mu^+)$ are pairwise distinct. 
\end{enumerate}
Recall that $M$ and $N$ are fixed, but first we observe:
	\begin{enumerate}[$(\star_1)$]
	\item W.l.o.g. we can assume that if $\alpha < \mu^+$ and $p(\bar{z}_\gamma) \subseteq \mrm{tp}_\Delta(\bar{c}_\alpha/M; N)$ has cardinality $< \xi$, \underline{then} $p(\bar{z}_\gamma)$ is realized in $M$.
\end{enumerate}
[Why? As by assumption $\mu = \mu^{< \xi}$, $\kappa < \xi$, $|\Delta| \leq 2^\kappa$, and so clearly $\bigcup \{\mrm{tp}_\Delta(\bar{c}_\alpha/M; N) : \alpha < \mu^+ \}$ (which is simply a set of formulas) has size $\leq \mu$. Notice that of course we can replace $M$ by $M'$ if $M \preccurlyeq_{\mathfrak{K}} M' \preccurlyeq_{\mathfrak{K}} N$ and $|M'| =\mu$, so $\Gamma = \{p : p \text{ as in $(\star_1)$} \}$ has cardinality $\leq \mu$. For each $p \in \Gamma$, choose $\alpha_p < \mu^+$ such that $p \subseteq \mrm{tp}_\Delta(\bar{c}_{\alpha_p}/M; N)$. Let $\alpha = \sup\{\alpha_p : p \in \Gamma\}$, so $\alpha < \mu^+$ and $M' = M_{\alpha+1}$ is as required.]
\begin{enumerate}[$(\star_2)$]
	\item For each $\alpha < \mu^+$,
we try to choose $(\bar{a}^\alpha_{(i, 1)}, \bar{a}^\alpha_{(i, 2)}, \bar{c}^\alpha_i, \mathfrak{m}^\alpha_{(i, 1)}, \mathfrak{m}^\alpha_{(i, 2)},  \psi^\alpha_{(i, 1)}, \psi^\alpha_{(i, 2)})$, by induction on $i < \xi$, such that the following happens:
	\begin{enumerate}[(a)]
	\item $\bar{a}^\alpha_{(i, \ell)} \in M^{\kappa}$, for $\ell = 1, 2$;
	\item for $\ell = 1, 2$, $\psi^\alpha_{(i, \ell)}(\bar{z}_\gamma, \bar{x}_\kappa) = \psi_{\mathfrak{m}^\alpha_{(i, \ell)}}(\bar{z}_\gamma, \bar{x}_\kappa)$, where
	$$\mathfrak{m}^\alpha_{(i, \ell)} = (M \restriction \bar{a}^\alpha_{(i, \ell)}, N, \bar{a}^\alpha_{(i, \ell)}, \bar{c}^\alpha_i) \in \mathcal{W}^{\mrm{large}}_{(\mathfrak{K}, \kappa, \gamma)};$$
	\item $N \models \psi^\alpha_{(i, \ell)}(\bar{c}^\alpha_i, \bar{a}^\alpha_{(i, \ell)})$, for $\ell = 1, 2$;
	\item $\psi^\alpha_{(i, 1)}$ and $\psi^\alpha_{(i, 2)}$ are contradictory and they are both from $\Delta$;
	\item $\bar{c}^\alpha_i \in M^\gamma$;
	\item if $\ell = 1, 2$ and $j \leq i < \xi$, then
	$$N \models \psi^\alpha_{(j, \ell)}(\bar{c}^\alpha_i, \bar{a}^\alpha_{(j, \ell)}) \wedge \psi^\alpha_{(j, \ell)}(\bar{c}_\alpha, \bar{a}^\alpha_{(j, \ell)})$$
	\item if $\ell = 1, 2$ and $j < i < \xi$, then we have that
	$$N \models \psi^\alpha_{(j, \ell)}(\bar{c}^\alpha_j, \bar{a}^\alpha_{(i, 1)}) \leftrightarrow \psi^\alpha_{(j, \ell)}(\bar{c}^\alpha_j, \bar{a}^\alpha_{(i, 2)}).$$
	\end{enumerate}
\end{enumerate}
\begin{enumerate}[$(\star_3)$]
	\item Let $i(\alpha)$ be the minimal $i \leq \xi$ such that the induction from $(\star_2)$ stops, so $i(\alpha) \leq \xi$ (recall that the induction from $(\star_2)$ is on $i < \xi$).
\end{enumerate}
\begin{enumerate}[$(\star_4)$]
	\item If for some $\alpha < \mu^+$ we have that $i(\alpha) = \xi$, then we get a contradiction to the assumption (4) which says that $\mathfrak{K}$ fails the syntactic $(\xi, \gamma_*, \Delta^+)$-order property.
\end{enumerate}
Why $(\star_4)$? Suppose that the assumption of $(\star_4)$ holds, i.e., $i(\alpha) = \xi$.  As by assumption we have that $\mrm{cf}(\xi) > |\Delta|$, then, for some $\psi_1, \psi_2$, the order type of $\mathcal{U}$ is equal to $\xi$, where:
$$\mathcal{U} = \{i < \xi : (\psi_1, \psi_2) = (\psi^\alpha_{(i, 1)}, \psi^\alpha_{(i, 2)})\}.$$
Thus, letting, for $i \in \mathcal{U}$, $\bar{b}^\alpha_i :=(\bar{c}^{\alpha}_i)^\frown(\bar{a}^\alpha_{{(i, 1)}})^\frown(\bar{a}^\alpha_{(i, 2)})$, which has length $\gamma_* = \gamma + \kappa + \kappa$, we have $(\bar{b}^\alpha_i : i \in \mathcal{U})$ exemplifies the syntactic $(\xi, \gamma_*, \Delta^+)$-order property. To see this, let $\psi_1 = \psi_{\mathfrak{m}_1}$ and $\psi_2 = \psi_{\mathfrak{m}_2}$. Notice now that
	\begin{enumerate}[$(\star_{4.1})$]
	\item If $j \leq i$ and $i, j \in \mathcal{U}$, then $N \models \delta^1_{(\mathfrak{m}_1, \mathfrak{m}_2)}(\bar{b}_i, \bar{b}_j)$.
	\end{enumerate}
[Why? Note now the following.
\begin{enumerate}[$(\cdot_1)$]
\item If $\ell = 1, 2$, then $N \models \psi_{\mathfrak{m}_\ell}(\bar{c}^\alpha_i, \bar{a}^\alpha_{(j, \ell)})$.
\newline [This is by $(\star_{2})$(f).]
\end{enumerate}
\begin{enumerate}[$(\cdot_2)$]
\item If $\ell = 1$, then $N \models \psi_{\mathfrak{m}_\ell}((\bar{b}_{i}(\beta) : \beta < \gamma), (\bar{b}_{j}(\gamma + \beta) : \beta < \kappa))$, 
\newline [This is by $(\cdot_1)$ and the choice of $\bar{b}_i$ and $\bar{b}_j$.]
\end{enumerate}
\begin{enumerate}[$(\cdot_3)$]
\item If $\ell = 2$, then $N \models \psi_{\mathfrak{m}_\ell}((\bar{b}_{i}(\beta) : \beta < \gamma), (\bar{b}_{j}(\gamma + \kappa + \beta) : \beta < \kappa))$.
\newline [This is by $(\cdot_1)$ and the choice of $\bar{b}_i$ and $\bar{b}_j$.]
\end{enumerate}
\begin{enumerate}[$(\cdot_4)$]
\item $N \models \delta^1_{(\mathfrak{m}_1, \mathfrak{m}_2)}(\bar{b}_i, \bar{b}_j)$.
\newline [This is by the definition of $\delta^1_{(\mathfrak{m}_1, \mathfrak{m}_2)}$ and $(\cdot_2)$, $(\cdot_3)$.]
\end{enumerate}
So $(\star_{4.1})$ holds indeed.]
	\begin{enumerate}[$(\star_{4.2})$]
	\item If $j < i$ and $i, j \in \mathcal{U}$, then $N \models \delta^2_{(\mathfrak{m}_1, \mathfrak{m}_2)}(\bar{b}_j, \bar{b}_i)$, i.e., $N \models \neg \delta^1_{(\mathfrak{m}_1, \mathfrak{m}_2)}(\bar{b}_j, \bar{b}_i)$.
	\end{enumerate}
[Why? Toward contradiction assume that $N \models \delta^1_{(\mathfrak{m}_1, \mathfrak{m}_2)}(\bar{b}_j, \bar{b}_i)$.
\begin{enumerate}[$(\cdot_1)$]
\item for $\ell = 1, 2$, $N \models \psi_{\mathfrak{m}_\ell}(\bar{c}^\alpha_j, \bar{a}^\alpha_{(i, 1)}) \leftrightarrow \psi^\alpha_{(j, \ell)}(\bar{c}^\alpha_j, \bar{a}^\alpha_{(i, 2)})$.
\newline [This is by $(\star_2)$(g).]
\end{enumerate}
\begin{enumerate}[$(\cdot_2)$]
\item for $\ell = 1, 2$, we have
\begin{align*}
\quad &N \models \psi_{\mathfrak{m}_\ell}((\bar{b}_{j}(\beta) : \beta < \gamma), (\bar{b}_{i}(\gamma + \beta) : \beta < \kappa)) \leftrightarrow \\
& \psi_{\mathfrak{m}_\ell}(((\bar{b}_{j}(\beta) : \beta < \gamma), (\bar{b}_{i}(\gamma + \kappa + \beta) : \beta < \kappa)))
\end{align*}
\newline [This is by $(\cdot_1)$ and the choice of $\bar{b}_i$ and $\bar{b}_j$.]
\end{enumerate}
\begin{enumerate}[$(\cdot_3)$]
\item $N \models \psi_{\mathfrak{m}_1}((\bar{b}_{j}(\beta) : \beta < \gamma), (\bar{b}_{i}(\gamma + \beta) : \beta < \kappa))$.
\newline [This is by our assumption toward contradiction and the definition of $\delta^1_{(\mathfrak{m}_1, \mathfrak{m}_2)}$.]
\end{enumerate}
\begin{enumerate}[$(\cdot_4)$]
\item $N \models \psi_{\mathfrak{m}_1}(((\bar{b}_{j}(\beta) : \beta < \gamma), (\bar{b}_{i}(\gamma + \kappa + \beta) : \beta < \kappa)))$.
\newline [This is by $(\cdot_2)$ and $(\cdot_3)$.]
\end{enumerate}
\begin{enumerate}[$(\cdot_5)$]
\item $N \models \neg \psi_{\mathfrak{m}_2}(((\bar{b}_{j}(\beta) : \beta < \gamma), (\bar{b}_{i}(\gamma + \kappa + \beta) : \beta < \kappa)))$.
\newline [By $(\cdot_4)$ and $\psi_{\mathfrak{m}_1}$, $\psi_{\mathfrak{m}_2}$ being contradictory.]
\end{enumerate}
But $(\cdot_5)$ contradicts our assumption toward contradiction, so $(\star_{4.2})$ holds indeed.]
	\begin{enumerate}[$(\star_{4.3})$]
	\item $\delta^1_{(\mathfrak{m}_1, \mathfrak{m}_2)}$ and $\delta^2_{(\mathfrak{m}_1, \mathfrak{m}_2)}$ are contradictory.
	\end{enumerate}
[Why? By thee choice of $\delta^1_{(\mathfrak{m}_1, \mathfrak{m}_2)}$ and $\delta^2_{(\mathfrak{m}_1, \mathfrak{m}_2)}$.]
\newline Together $(\star_{4.1})$-$(\star_{4.3})$ establish $(\star_4)$, so this ends the proof of $(\star_4)$.
\begin{enumerate}[$(\star_5)$]
	\item Thus we have that for every $\alpha < \mu^+$ we have that $i(\alpha) < \xi$.
\end{enumerate}
\begin{enumerate}[$(\star_6)$]
\item For some $\alpha_* < \mu^+$, $|\mathcal{V}| = \mu^+$, where:
	 $$\mathcal{V} = \{\beta < \mu^+ : i(\beta) = i(\alpha_*) \text{ and } \forall i < i(\alpha_*)\forall \ell \in \{1, 2\}, \, \bar{c}^{\alpha_*}_i = \bar{c}^\beta_i, \psi^\beta_{(i, \ell)} = \psi^{\alpha_*}_{(i, \ell)} \}.$$
\end{enumerate}
[Why? As the number of possible sequences
$$(i(\alpha), (\psi^\alpha_{(i, 1)}, \psi^\alpha_{(i, 1)}) : i < i(\alpha)), ((\bar{a}^\alpha_{(i, 1)}, \bar{a}^\alpha_{(i, 2)}, \bar{c}^\alpha_i) : i < i(\alpha))$$
is $\leq (\xi \times |\Delta| \times |\Delta| \times |M|^{\kappa} \times |M|^{\kappa} \times |M|^\gamma)^{< \xi} \leq (\mu^\kappa)^{< \xi} = \mu^{< \xi} = \mu$.]
\begin{enumerate}[$(\star_7)$]
	\item 
	\begin{enumerate}[(a)]
	\item The set $\{\mrm{tp}_\Delta(\bar{c}_\alpha, M, N) : \alpha \in \mathcal{V}\}$ has size $\leq 2^{2^{|i(\alpha_*)| + \kappa + |\Delta|}}$.
	\item as $|\mathcal{V}| = \mu^+$ and $2^{2^{|i(\alpha_*)| + \kappa + |\Delta|}}$ is $\leq \mu$ we get a contradiction. 
	\end{enumerate}
\end{enumerate}
Why $(\star_7)$? It suffices to prove $(\star_7)$(a). To this extent, let 
$$\mathbf{I} = \{\bar{b} : \bar{b} \in M^\kappa \text{ and } M \restriction \mrm{ran}(\bar{b}) \preccurlyeq_{\mathfrak{K}} M\}.$$
We then define
%
$$E = \{(\bar{b}_1, \bar{b}_2) \in \mathbf{I} \times \mathbf{I} : \text{ if } \ell =1, 2; j < i(\alpha_*) \text{ then } N \models \psi^{\alpha_*}_{(j, \ell)}(\bar{c}^{\alpha_*}_j, \bar{b}_1) \leftrightarrow \psi^{\alpha_*}_{(j, \ell)}(\bar{c}^{\alpha_*}_j, \bar{b}_2) \}.$$
Next we have that:
\begin{enumerate}[$(\star_{7.1})$]
	\item
	\begin{enumerate}[(a)]
	\item $E$ is an equivalence relation;
	\item $E$ has $\leq 2^{2 \times i(\alpha_*)}$ equivalence classes;
	\item $2 \times i(\alpha_*) < \xi$ (recalling the assumptions on $\xi$);
	\item if $\bar{b}_1 E \bar{b}_2$, $\psi(\bar{z}_\gamma, \bar{x}_\kappa) \in \Delta$ and $\alpha \in \mathcal{V}$, then $N \models \psi(\bar{c}_\alpha, \bar{b}_1) \leftrightarrow \psi(\bar{c}_\alpha, \bar{b}_2)$.
	\item for each $\alpha \in \mathcal{V}$, let
	$$\mathbf{Y}_\alpha = \{(Y, \psi) : Y \text{ is an $E$-equivalence class}, \psi \in \Delta \text{ and } \bar{b} \in Y \Rightarrow N \models \psi(\bar{c}_\alpha, \bar{b})\};$$
	\item if $\alpha \neq \beta \in \mathcal{V}$, then $\mathbf{Y}_\alpha \neq \mathbf{Y}_\beta$;
	\item $|\{\mathbf{Y}_\alpha : \alpha \in \mathcal{V}\}| \leq 2^{|M^\kappa/E| \cdot |\Delta|}$;
	\item if $\alpha < \beta < \mu^+$, and $\mathbf{Y}_\alpha = \mathbf{Y}_\beta$, then $\mrm{tp}_\Delta(\bar{c}_\alpha/M; N) = \mrm{tp}_\Delta(\bar{c}_\beta/M; N)$.
	\end{enumerate}
\end{enumerate}
We prove $(\star_{7.1})$. The items needing proofs are (d), (f) and (g). Concerning item (d), if not then, we have that for some $\alpha \in \mathcal{V}$, $\psi \in \Delta$ and $\bar{b}_1, \bar{b}_2 \in \mathbf{I}$, which are $E$-equivalent, we have that $N$ does not satisfy $\psi(\bar{c}_\alpha, \bar{b}_1) \leftrightarrow  \psi(\bar{c}_\alpha, \bar{b}_2)$, so we have
$$N \models \psi(\bar{c}_\alpha, \bar{b}_1) \wedge \neg \psi(\bar{c}_\alpha, \bar{b}_2) \text{ or } N \models \neg \psi(\bar{c}_\alpha, \bar{b}_1) \wedge \psi(\bar{c}_\alpha, \bar{b}_2).$$
As we can interchange the role of $\bar{b}_1$ and $\bar{b}_2$, w.l.o.g. we have that
$$N \models \psi(\bar{c}_\alpha, \bar{b}_1) \wedge \neg \psi(\bar{c}_\alpha, \bar{b}_2).$$
Now, as $\psi \in \Delta$ there is $\mathfrak{m}_1 \in \mathcal{W}^{\mrm{large}}_{(\mathfrak{K}, \kappa, \gamma)}$ such that $\psi = \psi_{\mathfrak{m}_1}$. Since $\bar{b}_2 \in \mathbf{I}$ we can find $\mathfrak{m}_2 \in \mathcal{W}^{\mrm{large}}_{(\mathfrak{K}, \kappa, \gamma)}$ such that $N \models \psi_{\mathfrak{m}_2}(\bar{c}_\alpha, \bar{b}_2)$, and obviously $\psi_{\mathfrak{m}_1}$ and $\psi_{\mathfrak{m}_2}$ are contradictory, so we get a contradiction to $i(\alpha) = i(\alpha_*)$, we elaborate on this last passage. To this extent, let
$$p(\bar{z}_\gamma) = \{\psi_{\mathfrak{m}_1}(\bar{z}_\gamma, \bar{b}_1), \psi_{\mathfrak{m}_1}(\bar{z}_\gamma, \bar{b}_2)\} \cup \{\psi^\alpha_{(j, \ell)}(\bar{z}_\gamma, \bar{a}^\alpha_{(j, \ell)}) : \ell \in \{ 1, 2\}, \, j < i(\alpha_*)\}.$$
Clearly, $p(\bar{z}_\gamma)$ is a subset of $\mrm{tp}_\Delta(\bar{c}_\alpha/M; N)$ of size $< \xi$, hence by $(\star_1)$, there is $\bar{c}_* \in M^\gamma$ realizing it. So in the inductive choice in $(\star_2)$ we could have chosen, for our $\alpha$ and $i = i(\alpha)$ from $(\star_3)$, $(\bar{a}^\alpha_{(i, 1)}, \bar{a}^\alpha_{(i, 2)}, \bar{c}^\alpha_i, \mathfrak{m}^\alpha_{(i, 1)}, \mathfrak{m}^\alpha_{(i, 2)},  \psi^\alpha_{(i, 1)}, \psi^\alpha_{(i, 2)})$ as follows:
\begin{enumerate}[$(\cdot)$]
\item $\bar{a}^\alpha_{(i, 1)} = \bar{b}_1$;
\item $\bar{a}^\alpha_{(i, 1)} = \bar{b}_2$;
\item $\bar{c}^\alpha_{i} = \bar{c}_*$;
\item $\mathfrak{m}^\alpha_{(i, 1)} = \mathfrak{m}_1$;
\item $\mathfrak{m}^\alpha_{(i, 2)} = \mathfrak{m}_2$;
\item $\psi^\alpha_{(i, 1)} = \psi_{\mathfrak{m}_1}$;
\item $\psi^\alpha_{(i, 2)} = \psi_{\mathfrak{m}_2}$.
\end{enumerate}
But this contradicts the choice of $i(\alpha) = i(\alpha_*)$. This proves item (d). 

\smallskip \noindent
Concerning item (f), let $\alpha \neq \beta \in \mathcal{V}$, so $\mrm{tp}_\Delta(\bar{c}_{\alpha}/M; N) \neq \mrm{tp}_\Delta(\bar{c}_{\beta}/M; N)$, by $(\star_{0})$. So there are $\psi(\bar{y}_\gamma, \bar{x}_\kappa) \in \Delta$ and $\bar{b} \in \mathbf{I}$ such that 
$$\psi(\bar{y}_\gamma, \bar{b}) \in \mrm{tp}_\Delta(\bar{c}_\alpha/M; N) \;\; \Leftrightarrow \;\; \psi(\bar{y}_\gamma, \bar{b}) \in \mrm{tp}_\Delta(\bar{c}_\beta/M; N).$$
As we can interchange the role of $\alpha$ and $\beta$, w.l.o.g. we have
$$N \models \psi(\bar{c}_\alpha, \bar{b}) \wedge \neg \psi(\bar{c}_\beta, \bar{b}).$$
Now, as $\psi \in \Delta$ there is $\mathfrak{m}_1 \in \mathcal{W}^{\mrm{large}}_{(\mathfrak{K}, \kappa, \gamma)}$ such that $\psi = \psi_{\mathfrak{m}_1}$. Also, as $\bar{b} \in \mathbf{I}$, we can find $\mathfrak{m}_2 \in \mathcal{W}^{\mrm{large}}_{(\mathfrak{K}, \kappa, \gamma)}$ such that $N \models \psi_{\mathfrak{m}_2}(\bar{c}_\alpha, \bar{b})$ and obviously $\psi_{\mathfrak{m}_1}$ and $\psi_{\mathfrak{m}_2}$ are contradictory. By the definition $\mathbf{Y}_\alpha$, we have $(\bar{b}/E, \psi_{\mathfrak{m}_1}) \in \mathbf{Y}_\alpha$ and similarly $(\bar{b}/E, \psi_{\mathfrak{m}_2}) \in \mathbf{Y}_\beta$, so necessarily $\mathbf{Y}_\alpha \neq \mathbf{Y}_\beta$, as promised. This proves item (f).

\smallskip \noindent Finally, concerning item (g), notice that $\mathbf{Y}_\alpha$ is a subset of 
$$Z = \{(\bar{b}/E, \psi) : \bar{b} \in \mathbf{I}, \, \psi \in \Delta\},$$
hence $\{\mathbf{Y}_\alpha : \alpha < \mu^+\} \subseteq \mathcal{P}(Z)$, so 
$$|\{\mathbf{Y}_\alpha : \alpha < \mu^+\}| \leq 2^{|Z|}.$$
Notice now that $|Z| = |\{\bar{b}/E : \bar{b} \in \mathbf{I} \}| \cdot |\Delta| \leq |M^\kappa/E| \cdot |\Delta|$, and so putting everything together we are done. This proves item (g). Lastly, item (h) follows.

\smallskip \noindent Thus, we have proved $(\star_{7.1})$, to conclude just notice that items (h) and (g) of $(\star_{7.1})$ imply $(\star_{7})$(a), and, as already observed, that leads to a contradiction.
\end{proof}

	\begin{remark} The proof of \ref{almost_stability_claim} is similar to \cite[Chapter I, Th. 1.10, pg. 277]{300}.
\end{remark}

	\begin{proof}[Proof of \ref{stability_th}] This follows from \ref{equivalence_syntactic}, \ref{almost_stability_claim}, and \cite[3.3]{977}, since it easily follows from \cite[3.3]{977} that the syntactic order property stated in \ref{almost_stability_claim} fails for any $\mrm{AEC}$ of $R$-modules.
\end{proof}

%
%

\section{A counterexample to Mazari-Armida's question}

	\begin{notation}\label{not_RP_module} Let $\mathcal{P}$ be a set of primes. We denote by $R_{\mathcal{P}}$ the sub-ring $\mathbb{Z}[\frac{1}{p} : p \in \mathcal{P}]$ of $\mathbb{Q}$. For $\mathcal{P} = \{p\}$ we may write $R_p$ also as $\mathbb{Z}[\frac{1}{p}]$.
\end{notation}

	\begin{definition}\label{def_P_torsion} Let $\mathcal{P}$ be a set of primes and $G \in \mrm{AB}$. We say that $G$ is $\mathcal{P}$-torsion when $G$ is torsion and, for any prime $p$, if $px = 0$ and $x \neq 0$, then $p \in \mathcal{P}$. 
	Given an abelian group $G$ and a prime $p$, we write $p_\ell^\infty \vert \, a$ to we mean that $p_\ell^n \vert \, a$ for all $n < \omega$.
\end{definition}

	\begin{proof}[Proof of Theorem~\ref{counterexample_th}] Let $\bar{p} = (p_1, ..., p_5)$ be distinct primes. 
		\begin{enumerate}[$(\star_1)$] 
	\item Fix $\lambda$ infinite and let $G_\lambda = \bigoplus \{R_{p_1} x_\alpha : \alpha < \lambda \}$ (recall Notation~\ref{not_RP_module}).
	\end{enumerate}
	\begin{enumerate}[$(\star_2)$] 
	\item For every $\mathcal{U} \subseteq \lambda$ we define $G^*_{\mathcal{U}} \in \mrm{TFAB}$ as follows:
	\begin{enumerate}[(a)]
	\item $G^0_{\mathcal{U}} = G_\lambda \oplus N \oplus H$, where:
	$$N = R_{p_2}y \text{ and } H = \bigoplus \{\mathbb{Z} z_\alpha: \alpha < \lambda\};$$
	\item $G^1_{\mathcal{U}} = \mathbb{Q}G_\lambda \oplus \mathbb{Q}N \oplus \mathbb{Q}H$;
	\item $G^*_{\mathcal{U}}$ is the subgroup of $G^1_{\mathcal{U}}$ generated by:
	\begin{enumerate}
	\item $p^{-n}_1 x_\alpha$, $\alpha < \lambda$, $n < \omega$;
	\item $p^{-n}_2 y$, $\alpha < \lambda$;
	\item $p^{-n}_3 z_\alpha$, $\alpha < \lambda$, $n < \omega$;
	\item $p^{-n}_4 (x_\alpha + z_\alpha)$, $\alpha \in \mathcal{U}$, $n < \omega$;
	\item $p^{-n}_5 (x_\alpha + y + z_\alpha)$, $\alpha \in \mathcal{U}$, $n < \omega$.
	\end{enumerate}
	\end{enumerate}
	\end{enumerate}
	\begin{enumerate}[$(\star_3)$] 
	\item We define $\mathbf{K} = \mathbf{K}(\bar{p})$ as the class of $G$ such that:
	\begin{enumerate}[(a)]
	\item $G$ is a torsion-free abelian group;
	\item for $\ell \in \{1, ..., 5\}$, we define $G[p_\ell] = \{a \in G : p_\ell^\infty \vert \, a\}$;
	\item $G[p_1]$ is an $\aleph_1$-free $R_{p_1}$-module (recall Notation~\ref{not_RP_module}), where we recall that $\aleph_1$-free means that every countable submodule is free;
	\item if $G \neq G[p_1]$, \underline{then} for some $a_\star$ we have:
	\begin{enumerate}[$(\cdot_1)$]
	\item $a_\star \in G[p_2] \setminus \{0 \}$;
	\item $G[p_2] = \langle a_\star \rangle^*_{G} \cong R_{p_2}a_\star$ (where $\langle a_\star \rangle^*_{G}$ denotes pure closure in $G$);
	\item inside $G$ the group $C :=G[p_1] \oplus G[p_2] \oplus G[p_3]$ is well-defined;
	\item $G/C$ is $\{p_4, p_5\}$-torsion;
	\item for some partial embedding $h$ from $G[p_1]$ onto $G[p_3]$ we have:
	$$\{(x, h(x)) : x \in \mrm{dom}(h)\} = \{(x, z) : x \in G[p_1], \, z \in G[p_3], \, p_4^\infty \vert \, (x + z)\}$$
	$$G[p_4] = \{(x, h(x)) : x \in \mrm{dom}(h))\};$$
	\item for $H_1 = \{x \in G[p_1] : \exists z \in G[p_3] \text{ such that } p_4^\infty \vert \, (x+z)\}$ we have:
	\begin{enumerate}[(i)]
	\item $H_1$ is an $\aleph_1$-free $R_{p_1}$-module;
	\item $G[p_1]/H_1$ is an $\aleph_1$-free $R_{p_1}$-module;
	\end{enumerate}
	\item for $H_3 = \{z \in G[p_3] : \exists x \in G[p_1] \text{ such that } p_4^\infty \vert \, (x+z)\}$ we have:
	\begin{enumerate}[(i)]
	\item $H_3$ is an $\aleph_1$-free $R_{p_1}$-module;
	\item $G[p_3]/H_3$ is an $\aleph_1$-free $R_{p_1}$-module;
	\end{enumerate}
	\item $G[p_5]$ is equal to $A$, where:
	$$A = \langle \{r(x + a_\star + h(x)) : r \in R_{p_5}, x \in H_1\} \rangle_G.$$
	\end{enumerate}
	\end{enumerate}
\end{enumerate}
\begin{enumerate}[$(\star_4)$] 
	\item $(\mathbf{K}, \leq_{\mrm{pure}})$ is an $\mrm{AEC}$.
\end{enumerate}
[Why? The coherence and unions of chains axioms follow from the fact that the class is axiomatized by purity conditions, which are preserved under unions of chains of pure subgroups and substructures. The L\"owenheim-Skolem condition holds because all the relevant existence demands involve countable objects, so we have an AEC.]
	\begin{enumerate}[$(\star_5)$] 
	\item $G_\lambda \in \mathbf{K}$ and, for every $\mathcal{U} \subseteq \lambda$, $G^*_{\mathcal{U}} \in \mathbf{K}$ and $G_\lambda \leq_{\mrm{p}} G^*_{\mathcal{U}}$.
\end{enumerate}
[Why? Easy.]
\begin{enumerate}[$(\star_6)$] 
	\item For $\mathcal{U} \subseteq \lambda$, let $t_{\mathcal{U}} = \gtp(y/G_\lambda; G^*_{\mathcal{U}})$.
\end{enumerate}
\begin{enumerate}[$(\star_7)$] 
	\item For $\mathcal{U} \neq \mathcal{V} \subseteq \lambda$, $t_{\mathcal{U}} \neq t_{\mathcal{V}}$.
\end{enumerate}
\begin{enumerate}[$(\star_8)$] 
	\item $(\mathbf{K}, \leq_{\mrm{pure}})$ is not $\lambda$-stable, for every $\lambda$.
\end{enumerate}
[Why? Follows from $(\star_5)$ and $(\star_7)$.]
\end{proof}


\begin{thebibliography}{10}

\bibitem{Bal_categoricity}
J. T. Baldwin.
\newblock {\em Categoricity}.
\newblock University Lecture Series, vol.~50, American Mathematical Society, Providence, RI, 2009.


\bibitem{BKV06}
J. T. Baldwin, D.W. Kueker, and M. VanDieren. 
\newblock {\em Upward stability transfer theorem for tame abstract elementary classes}.
\newblock Notre Dame J. Form. Log. {\bf 47} (2006), 291-298.


\bibitem{BS08} 
J. T. Baldwin and S. Shelah. 
\newblock {\em Examples of non-locality}.
J. Symb. Log. {\bf 73} (2008), 765-782.

\bibitem{baur}
W. Baur.
\newblock {\em $	aleph_0$-categorical modules}.
J. Symb. Log. {\bf 40} (1975), 213-220.

\bibitem{boney_tame}
W. Boney
\newblock {\em Tameness from large cardinal axioms}. 
\newblock J. Symb. Log. {\bf 79} (2014) , no. 4, 1092-1119.

\bibitem{boney}
W. Boney
\newblock {\em A module-theoretic introduction to abstract elementary classes}. 
\newblock Preprint, available on ArXiv.

\bibitem{tame_surv}
W. Boney and S. Vasey.
\newblock {\em  A survey on tame abstract elementary classes}-
\newblock In ``Beyond First Order Model Theory'' (José Iovino ed.), CRC Press (2017), 353–427.

\bibitem{fisher}
E. Fisher.
\newblock {\em  Abelian structures I}.
\newblock In ``Abelian Group Theory'' (David Arnold, Roger
Hunter, Elbert Walker, eds.), Springer Lecture Notes in Mathematics {\bf 616} (1977), 270-322.

\bibitem{gross}
R. Grossberg and M. VanDieren.
\newblock {\em Galois-stability for tame abstract elementary classes}.
\newblock J. Math. Log. {\bf 6} (2006), no. 1, 25-49.

\bibitem{armida_fuchs}
M. Mazari-Armida. 
\newblock {\em A model theoretic solution to a problem of L\'{a}szl\'{o}  Fuchs}. 
\newblock J. Algebra {\bf 567} (2021), 196-209.

\bibitem{armida2}
M. Mazari-Armida. 
\newblock {\em Characterizing categoricity in several classes of modules}. 
\newblock J.  Algebra {\bf 617} (2023), 382-401.

\bibitem{armida3}
M. Mazari-Armida. 
\newblock {\em On superstability in the class of flat modules and perfect rings}. 
\newblock Proc. Amer. Math. Soc. {\bf 149} (2021), 2639-2654.


\bibitem{armida}
M. Mazari-Armida. 
\newblock {\em Some stable non-elementary classes of modules}. 
\newblock J. Symb. Log. {\bf 88} (2023), no. 1, 93-117.

\bibitem{prest1}
M. Prest.
\newblock {\em Model Theory and Modules}. 
\newblock Cambridge University Press, 2009.

\bibitem{prest2}
M. Prest.
\newblock {\em Purity, Spectra and Localisation}. 
\newblock Cambridge University Press, 2013.

\bibitem{300}
S. Shelah. 
\newblock {\em Universal classes}. 
\newblock In: Classification theory (Chicago, IL, 1985), Vol. 1292, Springer, Berlin, pp. 264-418.

\bibitem{394}
S. Shelah. 
\newblock {\em Categoricity for abstract classes with amalgamation}. 
\newblock Ann. Pure Appl. Logic {\bf 98} (1999), no. 1-3, 261-294.

\bibitem{shelah_class}
S. Shelah.
\newblock {\em Classification Theory: and the Number of Non-Isomorphic Models}.
\newblock Elsevier, 1990.

\bibitem{shelah_abstr_ele_cla}
S. Shelah.
\newblock {\em Classification Theory for Abstract Elementary Classes}. 
\newblock College Publications, London, 2009.

\bibitem{932}
S. Shelah. 
\newblock {\em Maximal failures of sequence locality in $\mrm{AEC}$}. 
\newblock Preprint, available on ArXiv.

\bibitem{977}
S. Shelah.
\newblock {\em Modules and infinitary logics}.
\newblock In ``Groups and model theory'', Contemp. Math., vol.~576,
  Amer. Math. Soc., Providence, RI, 2012, pp.~305--316.
\newblock An updated version of this paper can be found on Shelah's archive:
  \url{https://shelah.logic.at/papers/977/}.

\bibitem{1184}
S. Shelah and A. Villaveces. 
\newblock {\em Infinitary logics and $\mrm{AEC}$}. 
\newblock Proc. Amer. Math. Soc. {\bf 150} (2022), 371-380.

%
\bibitem{vasey_2}
S. Vasey.
\newblock {\em Building independence relations in abstract elementary classes}. 
\newblock Ann. Pure Appl. Logic {\bf 167} (2016), 1029-1092.

\bibitem{vasey_lec_notes} 
S. Vasey. 
\newblock {\em MATH 269X -- Model theory for abstract elementary classes}.
\newblock Spring 2018, Lecture Notes, \url{svasey.com/academic-homepage-may-2020/aec-spring-2018/aec-lecture-notes_04_26_2018.pdf}.

\bibitem{zilber}
B.I. Zilber.
\newblock {\em A categoricity theorem for quasiminimal excellent classes}.
\newblock In ``Logic and its Applications'', Contemporary Mathematics, 297–306. AMS, 2005.

\end{thebibliography}
\end{document}